\documentclass[11pt]{amsart}

\usepackage[paper=letterpaper,margin=1in,twoside=false,includehead,headheight=18pt]{geometry}

\usepackage[utf8]{inputenc}

\usepackage{amsmath,amssymb,amsthm} 
\usepackage{enumerate,paralist}
\usepackage{graphicx}
\usepackage{mathtools}
\usepackage{todonotes}

\newtheorem{theorem}{Theorem}[section]
\newtheorem{lemma}[theorem]{Lemma} 
 
\newtheorem{corollary}[theorem]{Corollary}

{\theoremstyle{definition} 
\newtheorem{claim}{Claim} 
 
\newtheorem*{definition}{Definition} 
\newtheorem{conjecture}{Conjecture} 
 
 }

\newtheoremstyle{named}%
  {}{}						
  {\upshape}				
  {0pt}{\bfseries}			
  {.}						
  {.5em}					
  {\thmname{#1}\thmnote{ #3}}  

\theoremstyle{named}

\newcommand{\set}[1]{\left\{{#1}\right\}} 
\newcommand{\setof}[2]{\left\{{#1}\,:\,{#2}\right\}}
\newcommand{\of}{\subseteq}

\newcommand{\paste}{\wedge}
\newcommand{\brac}[1]{\left\langle{#1}\right\rangle} 

\DeclarePairedDelimiter{\abs}{\lvert}{\rvert}

\begin{document}

\title{A note on the values of independence polynomials at $-1$}
\date{\today}
\author{Jonathan Cutler} \address{Department of Mathematical Sciences\\
Montclair State University\\
Montclair, NJ} \email{jonathan.cutler@montclair.edu} 
\author{Nathan Kahl} \address{Department of Mathematics and Computer Science\\
Seton Hall University\\
South Orange, NJ} \email{nathan.kahl@shu.edu}
\maketitle

\begin{abstract}
	The \emph{independence polynomial $I(G;x)$} of a graph $G$ is $I(G;x)=\sum_{k=1}^{\alpha(G)} s_k x^k$, where $s_k$ is the number of independent sets in $G$ of size $k$.  The \emph{decycling number} of a graph $G$, denoted $\phi(G)$, is the minimum size of a set $S\of V(G)$ such that $G-S$ is acyclic.  Engstr\"om proved that the independence polynomial satisfies $\abs{I(G;-1)} \leq 2^{\phi(G)}$ for any graph $G$, and this bound is best possible.  Levit and Mandrescu provided an elementary proof of the bound, and in addition conjectured that for every positive integer $k$ and integer $q$ with $\abs{q}\leq 2^k$, there is a connected graph $G$ with $\phi(G)=k$ and $I(G;-1)=q$.  In this note, we prove this conjecture.
\end{abstract}

\section{Introduction} 
\label{sec:introduction}

Let $\alpha(G)$ denote the \emph{independence number of a graph $G$}, the maximum order of an independent set of vertices in $G$.  The \emph{independence polynomial of a graph $G$} is given by
\[
	I(G;x)=\sum_{k=1}^{\alpha(G)} s_k x^k,
\]
where $s_k$ is the number of independent sets of size $k$ in $G$.  The independence polynomial has been the object of much research (see for instance the survey \cite{LM05}).  One direction of this research, partly motivated by connections with hard-particle models in physics \cite{A13,BLN08,E09,HS10,J09}, has focused on the evaluation of the independence polynomial at $x=-1$.
 
The \emph{decycling number of a graph $G$}, denoted $\phi(G)$, is the minimum size of a set of vertices $S\of V(G)$ such that $G-S$ is acyclic.  Engstr\"om \cite{E09} proved the following bound on $I(G;-1)$, which is best possible.

\begin{theorem}[Engstr\"om]\label{thm:bound}
	For any graph $G$, $\abs{I(G;-1)}\leq 2^{\phi(G)}$.
\end{theorem}

Levit and Mandrescu \cite{LM11} gave an elementary proof of Theorem~\ref{thm:bound} and, in addition, proposed the following conjecture.

\begin{conjecture}[Levit and Mandrescu]\label{conj:main}
	Given a positive integer $k$ and an integer $q$ with $\abs{q}\leq 2^k$, there is a connected graph $G$ with $\phi(G)=k$ and $I(G;-1)=q$.
\end{conjecture}

For brevity, in this paper a graph $G$ with $\phi(G)=k$ and $I(G;-1)=q$, with $|q|\le 2^k$, will be referred to as a \emph{$(k,q)$-graph}.  In \cite{LM13}, Levit and Mandrescu provided constructions that gave $(k,q)$-graphs for all $k\le 3$ and $\abs{q}\leq 2^k$.  Also, they gave constructions for every $k$ provided $q\in \set{2^{\phi(G)},2^{\phi(G)}-1}$.  In this paper, we prove Conjecture \ref{conj:main}. 


\section{The Construction and Proof of Conjecture} 
\label{sec:tools}

The construction proceeds inductively, using particular $(k-1,q)$-graphs to produce the necessary $(k,q)$-graphs.  First we assemble the tools used in the construction.  The most important tool is a recursive formula for $I(G;x)$ due to Gutman and Harary \cite{GH}.  We let $N(v)=\setof{x\in V(G)}{xv\in E(G)}$ and $N[v]=\set{v}\cup N(v)$.  

\begin{lemma}\label{lem:recur}
	For any graph $G$ and any vertex $v\in V(G)$, 
	\[
		I(G;x)=I(G-v;x)+xI(G-N[v];x).
	\]
\end{lemma}

Using this, or simply counting independent sets, we can derive the independence polynomial at $-1$ for small graphs.  Some useful examples can be found in Table~\ref{tab:exs}.

\begin{table}[h]
\begin{center}
\begin{tabular}{c|c}
	$G$ & $I(G;-1)$\\ \hline \hline
	$K_1$ & $0$\\
	$K_2$ & $-1$\\
	$K_3=C_3$ & $-2$\\
	$C_6$ & $2$
\end{tabular}
\end{center}
\caption{Some small examples}
\end{table}

Since Lemma \ref{lem:recur} requires a particular vertex $v \in V(G)$ to be specified, it will often be helpful to root graphs for which we want to compute the independence polynomial at $-1$.  Given a graph $G$ and a vertex $v\in V(G)$, the \emph{rooted graph $G_v$} is the graph $G$ with the vertex $v$ labeled.  Of course, $I(G;-1)=I(G_v;-1)$ for any vertex $v\in V(G)$.

We now introduce two operations on rooted graphs which will be useful in our proof.  The first of these is called \emph{pasting}.

\begin{definition}
	Given two rooted graphs $G_v$ and $H_w$, the \emph{pasting of $G_v$ and $H_w$}, denoted $G_v\paste H_w$, is the rooted graph formed by identifying the roots $v$ and $w$.
\end{definition}

We note two important facts.  First, the pasting operation creates no new cycles, and thus $\phi(G_v \paste H_w) \le \phi(G_v)+\phi(H_w)$.   (In our construction the roots will be pendant vertices, and so $\phi(G_v \paste H_w) = \phi(G_v)+\phi(H_w)$.)  Second, if for two rooted graphs $G_v$ and $H_w$ the quantities $I(G_v;-1)$ and $I(H_w;-1)$ have been evaluated using Lemma \ref{lem:recur}, then the value of $I(G_v\paste H_w;-1)$ can be determined in a straightforward way.  It is well-known that, letting $G\cup H$ denote the disjoint union of $G$ and $H$, we have
\[
	I(G\cup H;x)=I(G;x)I(H;x).
\]
Deleting the pasted vertex in $G_v \paste H_w$ produces a disjoint union of graphs.  This fact, and the recurrences 
\begin{align*}
	I(G_v;-1)&=I(G_v-v;-1)-I(G_v-N[v];-1)\\
	I(H_w;-1)&=I(H_w-w;-1)-I(H_w-N[w];-1) 
\end{align*}
then give
\[
I(G_v\paste H_w;-1) = I(G_v-v;-1)I(H_w-w;-1)-I(G_v-N[v];-1)I(H_w-N[w];-1).
\]

It will be helpful to keep track of the various parts of the above calculation, and in order to do so we introduce the following bookkeeping device.  Given a rooted graph $G_v$, where $I(G_v-v;-1)=a$ and $I(G_v-N[v];-1)=b$, and hence $I(G_v;-1)=a-b$, we write $I(G_v;-1)=\brac{a-b,a,b}$ and say that $G_v$ has \emph{bracket} $\brac{a-b,a,b}$.  An example can be found in Figure~\ref{fig:one}.
\begin{figure}[ht]
	\includegraphics{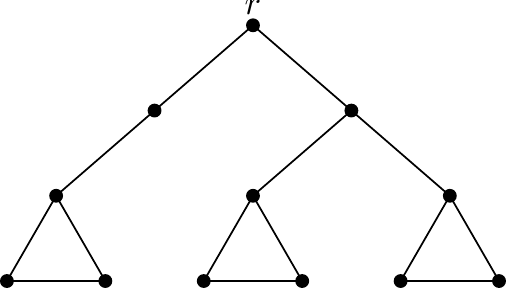}
	\caption{A graph rooted at $r$ with bracket $\brac{5,-3,-8}$.}\label{fig:one}
\end{figure}
Note that for a given rooted graph $G_v$ there are unique integers $a$ and $b$, determined by the root, with $I(G_v;-1)=\brac{a-b,a,b}$.  Using this notation, the calculations above give the following lemma.

\begin{lemma}[Pasting Lemma]\label{lem:paste}
	If $G_v$ and $H_w$ are rooted graphs on at least two vertices with $I(G_v;-1)=\brac{a-b,a,b}$ and $I(H_w;-1)=\brac{c-d,c,d}$, then 
\[	
	I(G_v\paste H_w;-1)=ac-bd = \brac{ac-bd,ac,bd}
\]
and $G_v \paste H_w$ has bracket $\brac{ac-bd,ac,bd}$.
\end{lemma}

Our second operation is a variation of the pasting operation which, however, is useful enough to merit its own terminology and notation.  

\begin{definition}
Given a rooted graph $G_v$ and an integer $k\geq 0$, the \emph{$\ell$-extension of $G_v$}, denoted $G_v^\ell$ is the graph formed by identifying the root $v$ with one of the endpoints of a (disjoint) path of length $\ell$ and reassigning the root to the other endpoint of the path.
\end{definition}

The length of a path is above measured in edges; for instance for a rooted graph $G_v$, the $0$-extension $G_v^0$ is simply $G_v$.  As with the pasting operation, no new cycles are created by the extension operation, and so here $\phi(G_v^{\ell}) = \phi(G)$ for any $\ell$.  In addition, the values of the independence polynomial at $-1$ of various extensions of a rooted graph $G_v$ are easy to characterize in terms of the bracket of $G_v$.  Indeed, extensions of $G_v$ have the same bracket values, up to sign, but in a different order.  The proof of the following lemma follows immediately from the recursion formula and is omitted.

\begin{lemma}[Extension Lemma]\label{lem:extend}
	If $G_v$ is a rooted graph with $I(G_v;-1)=\brac{a-b,a,b}$, then 
	\begin{align*}
	I(G_v^1;-1)&=\brac{-b,a-b,a}\\
	I(G_v^2;-1)&=\brac{-a,-b,a-b}
	\end{align*}
	and $I(G_v^3;-1) = \brac{b-a,-a-b}=-\brac{a-b,a,b} = -I(G_v;-1)$.
\end{lemma}

We illustrate the cycling phenomenon with $C_6$, a graph which will be used in our construction.  Obviously we may consider $C_6$ rooted at any given vertex. 

\begin{table}[h]
\begin{center}
\begin{tabular}{c|l}
	$\ell$ & $I(C_6^{\ell};-1)$\\ \hline \hline
	0 & $\brac{2,1,-1}$\\
	1 & $\brac{1,2,1}$\\
	2 & $\brac{-1,1,2}$\\
	3 & $\brac{-2,-1,1}$\\
	4 & $\brac{-1,-2,-1}$\\
	5 & $\brac{1,-1,-2}$\\
	6 & $\brac{2,1,-1}$
\end{tabular}
\end{center}
\caption{Brackets of $C_6^{\ell}$}\label{tab:exs}
\end{table}

 (Since $C_3$ has the same set of six brackets, in a different order, when extended, $C_3$ could also have been used in the constructions and proofs to come.  We choose $C_6$ solely because $C_6^0$ and $C_6^1$ have positive $I(G;-1)$.) 

Using the pasting and extension operations we have our final lemma, which shows that the word ``connected'' in the conjecture is superfluous.  Any disconected $(k,q)$-graph can be pasted together and extended to produce a connected $(k,q)$-graph.

\begin{lemma}\label{lem:nocon}
Let $G$ and $H$ be disjoint $(k_1,q_1)$ and $(k_2,q_2)$-graphs, respectively, with $k_1+k_2=k$ and $q_1q_2=q$.  Then there is a connected $(k,q)$-graph $F$, i.e., $F$ is connected, $\phi(F)=k_1+k_2=k$, and $I(F;-1) = q_1q_2 = I(G\cup H;-1)$.
\end{lemma}

\begin{proof}
Root the given graphs as $G_v$ and $H_w$ and let the corresponding brackets be $I(G_v;-1)=\brac{q_1,a,b}$ and $I(H_w;-1)=\brac{q_2,c,d}$, respectively.  Let $F' = (G_v^2 \paste H_w^2)^1$.  By the Extension Lemma,  $I(G_v^2;-1) = \brac{-a,-b,q_1}$ and $I(H_w^2;-1) = \brac{-c,-d,q_2}$.  Then, by the Pasting Lemma, 
\[
I(G_v^2\paste H_w^2;-1) = \brac{bd-q_1q_2,bd,q_1q_2}
\]
Therefore, again using the Extension Lemma, 
\begin{align*}
I(F';-1) &= I((G_v^2\paste H_w^2)^1;-1) \\
&= \brac{-q_1q_2,bd-q_1q_2,bd} \\
&= -q_1q_2 \\
&= -I(G\cup H;-1).
\end{align*}
In addition, neither the pasting nor extension operations produce cycles, so $\phi(F)=k_1+k_2=k=\phi(G\cup H)$.  

Now let $F = (F_x^{\prime 2} \cup K_2^2)^1$, where $F'_x$ is a rooted version of the graph $F'$ previously.  Then by the same analysis as above, we have $I(F;-1) = -I(F' \cup K_2;-1) = -(-1)I(F';-1) = I(G \cup H;-1)$, and $\phi(F) = \phi(F') = \phi(G\cup H)$, as required.  
\end{proof}

By setting $H=K_2$ and $H=C_6$ in Lemma \ref{lem:nocon} in turn, we obtain the following facts, which will also be useful in the proof.  These two facts were also noted by Levit and Mandrescu \cite{LM13}, who used different \emph{ad hoc} techniques in their constructions of the necessary graphs.

\begin{corollary}\label{cor:twice}
If $G$ is a $(k,q)$-graph then there exists (a) a connected $(k+1,2q)$-graph and (b) a connected $(k,-q)$-graph.  \end{corollary}

We now prove Conjecture~\ref{conj:main}.

\begin{theorem}\label{thm:main}
	Given a positive integer $k$ and an integer $q$ with $\abs{q}\leq 2^k$, there is a connected graph $G$ with $\phi(G)=k$ and $I(G;-1)=q$.
\end{theorem}

\begin{proof}
By Lemma \ref{lem:nocon} we do not need to produce connected $(k,q)$-graphs for all $\abs{q}\le 2^k$; disconnected $(k,q)$-graphs will suffice.  Since $I(G\cup K_1;-1)=0$ for all $G$, we can consider the case $q=0$ done for all $k$.  

As mentioned previously, our proof proceeds inductively on $k$.  When $k=1$ then $I(C_6;-1)=\brac{2,1,-1}$ and, as noted in Table~\ref{tab:exs}, by taking extensions of $C_6$, we rotate through all of $\{2,1,-1,-2\}$.  Thus the theorem is true for $k=1$.

For the induction step, assume $(k-1,q)$-graphs are constructible for all $q \le 2^{k-1}$.  By Corollary~\ref{cor:twice}(a) we immediately have that $(k,q)$-graphs for even $q \le 2^k$ are constructible.  By Corollary~\ref{cor:twice}(b) we also need only construct $(k,q)$-graphs for positive $q \le 2^k$.  It only remains, then, to construct $(k,q)$-graphs for $q$ each odd integer in $[0,2^k]$.   To that end, we prove the following claim.

\begin{claim}
	For each odd integer $q\in [0,2^k]$, there is a connected $(k,q)$-graph $G_v$ such that either $I(G_v;-1)=\brac{q,2^k,2^k-q}$ or $I(G_v;-1)=\brac{q,-2^k+q,-2^k}$.
\end{claim}

\begin{proof}
   For $k=1$, we see that the bracket of $C_6^1$ has the necessary form, i.e. $I(C_6^1;-1)=\brac{1,2,1}$.  Assume that the hypothesis of the claim is true for $k-1$; we seek to produce $(k,q)$-graphs for each odd $q\in [0,2^k]$ such that $2^k$ or $-2^k$ appears in their bracket.  We consider two cases:  $q\in [2^{k-1},2^k]$ and $q\in [0,2^{k-1}]$.

	For the first case, let $q$ be an odd integer in $[2^{k-1},2^k]$.  Necessarily then, $q=2^k-r$ for some $r\in [0,2^{k-1}]$.  By the induction assumption, there is some $(k-1,r)$-graph $G_v$ such that either $I(G_v;-1)=\brac{2^{k-1}-r,2^{k-1},r}$ or $I(G_v;-1)=\brac{2^{k-1}-r,-r,-2^{k-1}}$.  By the Pasting Lemma, then, $I(G_v\paste C_6^1;-1)=\brac{2^k-r,2^k,r}=q$ if the bracket of $G_v$ is of the first form, or $I(G_v\paste C_6^2;-1)=\brac{2^k-r,-r,-2^{k}}$ if the bracket of $G_v$ is of the second form.  Thus the claim is true for all $q\in [2^{k-1},2^k]$.

   We are left with the second case of the odd $q\in [0,2^{k-1}]$.  However, because $2^k$ appears in all the brackets in the previous case, necessarily these odd $q\in [0,2^{k-1}]$ correspond to the $r$ that appeared in those brackets.  Hence extending the constructions for the odd $q\in[2^{k-1},2^k]$ appropriately will produce these $r$. 
\end{proof}

The proof of the claim completes the induction, and completes the proof. 
\end{proof}


\section{Acknowledgment} 

The authors would like to thank Hannah Quense and Tara Wager for helpful discussions.

\bibliographystyle{amsplain}
\bibliography{indpoly2}

\providecommand{\bysame}{\leavevmode\hbox to3em{\hrulefill}\thinspace}
\providecommand{\MR}{\relax\ifhmode\unskip\space\fi MR }
\providecommand{\MRhref}[2]{%
  \href{http://www.ams.org/mathscinet-getitem?mr=#1}{#2}
}
\providecommand{\href}[2]{#2}
\begin{thebibliography}{1}

\bibitem{A13}
Micha{\l} Adamaszek, \emph{Special cycles in independence complexes and
  superfrustration in some lattices}, Topology Appl. \textbf{160} (2013),
  no.~7, 943--950. \MR{3037886}

\bibitem{BLN08}
Mireille Bousquet-M{\'e}lou, Svante Linusson, and Eran Nevo, \emph{On the
  independence complex of square grids}, J. Algebraic Combin. \textbf{27}
  (2008), no.~4, 423--450. \MR{2393250 (2009k:05137)}

\bibitem{E09}
Alexander Engstr{\"o}m, \emph{Upper bounds on the {W}itten index for
  supersymmetric lattice models by discrete {M}orse theory}, European J.
  Combin. \textbf{30} (2009), no.~2, 429--438.

\bibitem{GH}
Ivan Gutman and Frank Harary, \emph{Generalizations of the matching
  polynomial}, Utilitas Math. \textbf{24} (1983), 97--106.

\bibitem{HS10}
Liza Huijse and Kareljan Schoutens, \emph{Supersymmetry, lattice fermions,
  independence complexes and cohomology theory}, Adv. Theor. Math. Phys.
  \textbf{14} (2010), no.~2, 643--694. \MR{2721658 (2011i:81152)}

\bibitem{J09}
Jakob Jonsson, \emph{Hard squares with negative activity on cylinders with odd
  circumference}, Electron. J. Combin. \textbf{16} (2009), no.~2, Special
  volume in honor of Anders Bjorner, Research Paper 5, 22. \MR{2515768
  (2010g:05267)}

\bibitem{LM05}
Vadim~E. Levit and Eugen Mandrescu, \emph{The independence polynomial of a
  graph---a survey}, Proceedings of the 1st {I}nternational {C}onference on
  {A}lgebraic {I}nformatics, Aristotle Univ. Thessaloniki, Thessaloniki, 2005,
  pp.~233--254. \MR{2186466 (2006k:05163)}

\bibitem{LM11}
\bysame, \emph{A simple proof of an inequality connecting the alternating
  number of independent sets and the decycling number}, Discrete Math.
  \textbf{311} (2011), no.~13, 1204--1206.

\bibitem{LM13}
\bysame, \emph{The cyclomatic number of a graph and its independence polynomial
  at {$-1$}}, Graphs Combin. \textbf{29} (2013), no.~2, 259--273.

\end{thebibliography}

\end{document}